\documentclass[12pt]{article}
\usepackage{amssymb}
\usepackage{amsthm}
\usepackage{epsf,epsfig,amsfonts,amsmath}
\usepackage{graphicx}
\usepackage{a4wide}
\newtheorem{theorem}{Theorem}

\newtheorem{corollary*}{Corollary}

\newcommand{\be}{\begin{equation}}
\newcommand{\ee}{\end{equation}}

\newcommand{\weg}[1]{}   
\begin{document}

\title{Special Einstein's equations on K\"ahler manifolds}
\author{I. Hinterleitner, V. Kiosak}

\maketitle
\begin{abstract}
This work is devoted to the study of Einstein equations with a special shape of the energy-momentum tensor. Our results continue Stepanov's classification of Riemannian mani\-folds according to special properties of the energy-momentum tensor to K\"ahler manifolds. We show that in this case the number of classes reduces.

\end{abstract}



{\bf subclass}: 53B20; 53B30; 53B35; 53B50; 32Q15; 35Q76
%
\def\a{\alpha}
\def\b{\beta}
\def\epsilon{\varepsilon}
\def\s{\sigma}
\def\vn{$V_n$}
\def\vnn{$\bar V_n$}
\def\nad#1#2{\buildrel{#1} \over{#2}\!\!\strut}
\def\pod#1#2{\mathrel{\mathop{#2}\limits_{#1}}\strut}

\section{ Introduction }

The geometric properties of (pseudo-) Riemannian manifolds $V_n$,
 depending on the form the Einstein equations acquire in them,
 were studied by many authors. A large number of papers is  devoted to
 the study of  Einstein's equations with certain restrictions on the
 energy-momentum tensor and its first covariant derivatives
\cite{ha,pe,re,ss,st}.

S.E. Stepanov \cite{st1,st2} classified space-time manifolds
according to certain relations among the first covariant derivatives of the
energy-momentum tensor. He found three fundamental classes, related to geometrical
assumptions about space-time. By combinations of the conditions determining the three
fundamental classes he found three further classes. A seventh class is characterised
by the vanishing of the covariant derivative of the energy-momentum tensor.

In the present paper we partially take over Stepanov's classification to K\"ahler spaces
and investigate analogous, generalised classifying conditions. We show that for
two out of the three fundamental classes space-time is Ricci symmetric and the
energy-momentum tensor is covariantly constant.

In consequence, the energy-momentum tensor is covariantly constant
also for the three classes derived from the fundamental ones. Thus for K\"ahler spaces
 the number of classes of Einstein equations reduces to one with covariantly
 constant and one with non-constant energy-momentum tensor.
 We study some of their properties and generalisations.

All geometric objects are formulated locally under the assumption of
sufficient smoothness. Whereas S.E. Stepanov formulated his
classifications by making use of bundles, for our purpose it is
sufficient to write down the classifying relations in form of tensor
equations.



\section{Einstein's equations}

The equation of the following form:
\begin{equation}\label{F1}
R_{ij} - \frac{1}{2}\ R\, g_{ij}= T_{ij},
\end{equation}
is called {\it Einstein's equation}. Here $R_{ij}$ is the Ricci tensor  on the
manifold $V_n$, $g_{ij}$ is the metric tensor, $R$ is the scalar curvature,
and $T_{ij} $ is the energy-momentum tensor.

From the Bianchi identities of the Ricci tensor follows
 $ T_{\alpha i},_\beta g^{\beta\alpha} = 0 $,
(where the comma denotes the covariant derivative with
respect to a connection on the manifold $V_n$), and $g^{ij}$
are elements of the inverse matrix to $g_{ij}$.

Stepanov distinguishes the following three fundamental
types of manifolds in terms of covariant derivatives of the energy-momentum tensor:
\begin{eqnarray}
\Omega_1:&           &T_{ij,k} + T_{jk,i} +  T_{ki,j} = 0,\hfill\label{F3}\\[2mm]
\Omega_2:&          &T_{ij,k} - T_{ik,j} = 0,\hfill\label{F2}\\[2mm]
\Omega_3:&            &T_{ij,k} = a_kg_{ij} + b_ig_{jk} + b_jg_{ik},\hfill\label{F4}
\end{eqnarray}
where $a_k $ and $b_i$ are arbitrary vectors.\\
In space-time manifolds of type $\Omega_1$ the scalar curvature is
covariantly constant and the Ricci tensor is a Killing tensor,
i.e. $R_{ij}\frac{dx^i}{ds}\frac{dx^j}{ds}$ is constant  along geodesic
 curves with parameter $s$.

In the case $\Omega_2$ the scalar curvature is constant, too,
and the Levi-Civita connection of the metric, considered as a connection on the
 tangent bundle $TM$ satisfies the conditions of a Yang-Mills potential.

$\Omega_3$ is a slight generalisation in comparison with the
condition in \cite{st1,st2} on $R_{ij,k}$, reformulated in terms of
$T_{ij,k}$ the original conditions of Stepanov characterise
manifolds with non-constant curvature that admit non-trivial
geodesic mappings.

In \cite{st1,st2} three further classes are derived by
simultaneously imposing conditions $\Omega_1$ and $\Omega_2$,
$\Omega_2$ and $\Omega_3$, and $\Omega_1$ and $\Omega_3$,
respectively.

Using a generalized form of the introduced dependencies, we are going to study
manifolds characterised by the following conditions:
\begin{eqnarray}\label{F5}
\Omega_1^*:&            &T_{ij,k} + T_{jk,i} +  T_{ki,j} = \lambda_k T_{ij} + \lambda_i T_{jk} +  \lambda_j T_{ki}
+ \mu_k g_{ij} + \mu_i g_{jk} + \mu_j g_{ik},\\[2mm]
\label{F6}
\Omega_2^*:&            &T_{ij,k} - T_{ik,j} = \rho_k T_{ij}- \rho_jT_{ik} + \sigma_kg_{ij} -\sigma_jg_{ik} , \\[2mm]
\label{F7}
\Omega_3^*:&            &T_{ij,k} = \phi_k T_{ij} + \gamma_i T_{jk} +  \gamma_j T_{ki} + \eta_k g_{ij} + \chi_i g_{jk} + \chi_j g_{ik},
\end{eqnarray}
where $\phi_i $, $\lambda_i$, $\mu_i$, $\rho_i $,$\gamma_i $, $\eta_i$,
$\sigma_i $ and $\chi_i$ are arbitrary vectors.

\section{K\"ahler spaces}
An $n$-dimensional (pseudo-)Riemannian manifold $(M_n,g)$ is called
a \emph{K\"ahler space} $K_n$ if besides the metric tensor $g$, a \emph{structure} $F$,
 which is an affinor (i.e. a tensor field of type $(1,1)$),
is given on $M_n$ such that the following holds \cite{Miho,mkv,yano}:
\begin{equation}\label{F8}
 F^h_\alpha F^\alpha_i = -\delta^h_i; \quad F^\alpha_i g_{\alpha j} + F^\alpha_j g_{\alpha i}= 0; \quad F^h_{i,j} = 0 ,
\end{equation}
where $\delta_j^i$ is the Kronecker symbol.

Making use of this we can show that
\begin{equation}\label{F9}
 g_{ij} = g_{\alpha\beta} F^\alpha_i F^\beta_j ;\quad
 R_{ij} = R_{\alpha\beta} F^\alpha_i F^\beta_j.
\end{equation}
Then due to  \eqref{F1}, for the energy-momentum tensor the following relation holds
\begin{equation}\label{F10}
T_{ij} = T_{\alpha\beta} F^\alpha_i F^\beta_j;\quad
F^\alpha_i T_{\alpha j} + F^\alpha_j T_{\alpha i}= 0.
\end{equation}
We prove the following theorem.
\begin{theorem} 
If in a K\"ahler space holds the condition $\Omega_2^*$ or $\Omega_3^*$,
then the energy-momentum tensor satisfies
\begin{equation}\label{F11}
               T_{ij,k} = \rho_k T_{ij} + \sigma_k g_{ij}.
\end{equation}
\end{theorem}
\begin{proof}
Assume that in a K\"ahler space $K_n$ the condition (6) holds,
 multiply it by $F^i_l F^j_h$, contract with respect to $i$ and $j$
 and exchange $l$ for $i$ and $h$ for $j$. We obtain
\begin{equation}\label{F12}
           T_{\alpha\beta,k} F^\alpha_iF^\beta_j - T_{\alpha k,\beta} F^\alpha_iF^\beta_j = \rho_k T_{\alpha\beta}F^\alpha_iF^\beta_j- \rho_\beta T_{\alpha k}F^\alpha_iF^\beta_j +
            \sigma_kg_{\alpha\beta}F^\alpha_iF^\beta_j -\sigma_\beta g_{\alpha k}F^\alpha_iF^\beta_j .
\end{equation}

With the aid of (9) and (10) we can rewrite the last equation in the form
\begin{equation}\label{F13}
           T_{ij,k} - T_{\alpha k,\beta} F^\alpha_iF^\beta_j = \rho_k T_{ij}- \rho_\beta T_{\alpha k}F^\alpha_iF^\beta_j +
            \sigma_kg_{ij} -\sigma_\beta g_{\alpha k}F^\alpha_iF^\beta_j .
\end{equation}
After symmetrization of the indices $i$ and $k$ we get
\begin{equation}\label{F14}
           T_{ij,k} + T_{j k,i} = \rho_k T_{ij}+ \rho_i T_{j k} +
            \sigma_k g_{ij} -\sigma_i g_{j k} .
\end{equation}
Exchanging the indices $i$ and $j$ we obtain
\begin{equation}\label{F15}
           T_{ij,k} + T_{ik,j} = \rho_k T_{ij}+ \rho_j T_{i k} + \sigma_kg_{ij} +\sigma_j g_{i k} .
\end{equation}

Addition of (15) and (13) gives (11). Note that spaces satisfying
 $\Omega_3^*$  satisfy also the condition $\Omega_2^*$ as can be seen, when
\begin{equation}\label{F16}
    \rho_i = \phi_i - \gamma_i ;\quad
\sigma_i = \eta_i - \chi_i
 \end{equation}
 holds.
\end{proof}
By analyzing this result it is not difficult to prove
\begin{theorem} 
K\"ahler spaces $K_n$ belonging to class $\Omega_2$ or $\Omega_3$
are characterized by the following conditions
\begin{equation}\label{F17}
               T_{ij,k} = 0 ,\quad R_{ij,k} = 0.
 \end{equation}
\end{theorem}
From this theorem it follows immediately that for K\"ahler
 spaces also in the derived cases ($\Omega_1$ and $\Omega_2$, $\Omega_2$
 and $\Omega_3$, $\Omega_1$ and $\Omega_3$) the energy-momentum tensor is covariantly
 constant. So all the classes of Einstein equations, with the exception of
  $\Omega_1$, can be summarised under the characterisation $T_{ij,k}=0$.
From this follows that for K\"ahler spaces $K_n$ of class $\Omega_i$
(respectively $\Omega_i^*$) only those fulfilling condition (2) (resp. (5)) are relevant.

In a further step of generalisation we consider K\"ahler spaces characterised by
the following conditions
\begin{eqnarray}  \label{F18}
\Omega_4^*:&            T_{ij,k} - T_{ik,j} &= \rho_k T_{ij}- \rho_jT_{ik} + \sigma_kg_{ij} -\sigma_jg_{ik} +\rho_\alpha T_{i\beta}F^\alpha_k F^\beta_j \\[2mm]
&  &\ -\
\rho_\beta T_{i\alpha}F^\alpha_k F^\beta_j + \sigma_\alpha g_{i\beta}F^\alpha_k F^\beta_j -\sigma_\beta g_{i\alpha}F^\alpha_k F^\beta_j  .\nonumber
\\[3mm]
\label{F19}
\Omega_5^*:&            T_{ij,k} &= \phi_k T_{ij} + \gamma_i T_{jk} +  \gamma_j T_{ki} + \eta_k g_{ij} + \chi_i g_{jk} + \chi_j g_{ik} \\[2mm]
&  &\ +\
\gamma_\alpha T_{\beta k}F^\alpha_iF^\beta_j +  \gamma_\beta T_{k\alpha}F^\alpha_iF^\beta_j  + \chi_\alpha g_{\beta k}F^\alpha_iF^\beta_j + \chi_\beta g_{\alpha k}F^\alpha_iF^\beta_j. \nonumber
\end{eqnarray}
Applying the methods used in the proof of Theorem 1 to (18) and taking
into account (8), (9), (10) we convince ourselves that (18) acquires
the form (19), this proofs the next theorem
\begin{theorem} 
There are no K\"ahler spaces $K_n$ in the class $\Omega_4^*$ other
than spaces belonging to~$\Omega_5^*$.
\end{theorem}
In this way the K\"ahler spaces with non-constant energy-momentum tensor,
 considered in this paper, are divided into two essential classes: $\Omega_1^*$ and
 $\Omega_5^*$.

{\bf Acknowledgments.} This work was partially supported by
the Ministry of Education, Youth and Sports of the Czech Republic,
research \&\ development, project No. 0021630511.

\end{document}